\documentclass{article}
\usepackage{mystyle}
\begin{document}
\title{Strongly compact cardinals and ordinal definability}
\author{Gabriel Goldberg}
\maketitle
\begin{abstract}
    This paper explores several topics related to Woodin's HOD conjecture.
    We improve the large cardinal hypothesis of Woodin's HOD dichotomy theorem
    from an extendible cardinal to a strongly compact cardinal.
    We show that assuming there is a strongly compact cardinal and the
    \HOD\ hypothesis holds, there is no elementary embedding from \HOD\ to \HOD,
    settling a question of Woodin. We show that the \HOD\ hypothesis
    is equivalent to a uniqueness property of elementary embeddings
    of levels of the cumulative hierarchy. We prove
    that the \HOD\ hypothesis holds if and only
    if every regular cardinal above the first strongly compact cardinal
    carries an ordinal definable \(\omega\)-J\'onsson algebra.
    We show that if \(\kappa\) is supercompact,
    the \HOD\ hypothesis holds, and \HOD\ satisfies the Ultrapower Axiom,
    then \(\kappa\) is supercompact in \HOD.
\end{abstract}
\section{Introduction}
The Jensen covering theorem \cite{Jensen} states that
if \(0^\#\) does not exist, then every uncountable set of
ordinals is covered by a constructible set of the same cardinality.
This leads to a strong dichotomy for the cardinal correctness of the constructible
universe \(L\):
\begin{thm*}[Jensen]
    Exactly one of the following holds:
    \begin{enumerate}[(1)]
        \item For all singular cardinals \(\lambda\), \(\lambda\) is singular in \(L\)
        and \((\lambda^+)^L = \lambda^+\).
        \item Every infinite cardinal is strongly inaccessible in \(L\). 
    \end{enumerate}
\end{thm*}
The proof of this theorem and its generalizations to larger canonical inner models
tend to make heavy use of the special structure of such models. By completely different techniques,
however,
Woodin \cite{Woodin} showed that such a dichotomy holds
for the (noncanonical) inner model \(\HOD\) under large cardinal hypotheses.
\begin{thm*}[Woodin's \HOD\ dichotomy]
    If \(\kappa\) is extendible, exactly one of the following holds:
    \begin{enumerate}[(1)]
        \item \label{item:HOD_hypothesis} For all singular cardinals \(\lambda > \kappa\), \(\lambda\) is singular in \(\HOD\) and \((\lambda^+)^\HOD = \lambda^+\).
        \item Every regular cardinal greater than
        or equal to \(\kappa\) is measurable in \(\HOD\). 
    \end{enumerate}
\end{thm*}
Woodin's HOD conjecture states that conclusion \ref{item:HOD_hypothesis},
which in this context is known as the \textit{\HOD\ hypothesis},\footnote{
    In the context of bare ZFC (without large cardinal axioms),
    the \HOD\ hypothesis states that there is a proper class of regular
    cardinals that are not \(\omega\)-strongly measurable in \HOD. Conceivably
    the \HOD\ hypothesis is provable in ZFC.
}
is provable from large cardinal axioms. 

The first few theorems of this paper (\cref{section:hod_dichotomy}) show that the 
HOD dichotomy in fact takes hold far
below the least extendible cardinal.
\begin{thm*}
    If \(\kappa\) is strongly compact, exactly one of the following holds:
    \begin{enumerate}[(1)]
        \item For all singular cardinals \(\lambda > \kappa\), 
        \(\lambda\) is singular in \(\HOD\) and \((\lambda^+)^\HOD = \lambda^+\).
        \item All sufficiently large regular cardinals are measurable in \(\HOD\). 
    \end{enumerate}
\end{thm*}
We prove this as a corollary of a theorem
establishes a similar dichotomy for a fairly broad class of inner models.
An inner model \(N\) of ZFC is \textit{\(\omega\)-club amenable} if
for all ordinals \(\delta\) of uncountable cofinality, the \(\omega\)-club filter
\(F\) on \(\delta\) satisfies \(F\cap N \in N\).
\begin{repthm}{cor:club_amenable_weak_cover}
    Assume \(N\) is \(\omega\)-club amenable and
     \(\kappa\) is \(\omega_1\)-strongly compact.
    Either all singular cardinals \(\lambda > \kappa\)
    are singular in \(N\) and \((\lambda^+)^N = \lambda^+\), or
    all sufficiently large regular cardinals are measurable in \(N\). 
\end{repthm}
Note that we use a somewhat weaker large cardinal hypothesis than 
strong compactness: a cardinal \(\kappa\) is \(\omega_1\)-strongly compact
if every \(\kappa\)-complete filter extends to an \(\omega_1\)-complete
ultrafilter.

The main theorem of \cite{UniqueEmbeddings}
states that if \(j_0,j_1 : V\to M\) are elementary embeddings,
then \(j_0\restriction \Ord = j_1\restriction \Ord\).
\cref{section:uniqueness} turns to the relationship between local
forms of this theorem and Woodin's HOD conjecture.
\begin{repthm}{thm:extendible_uniqueness}
    Assume \(\kappa\) is an extendible cardinal.
    Then the following are equivalent:
    \begin{enumerate}[(1)]
        \item The \HOD\ hypothesis.
        \item For all regular \(\delta\geq \kappa\) and
        all sufficiently large \(\alpha\), any elementary embeddings
        \(j_0,j_1 : V_\alpha\to M\) with
        \(j_0(\delta) = j_1(\delta)\) and \(\sup j_0[\delta] = \sup j_1[\delta]\)
        agree on \(\delta\).
    \end{enumerate}
\end{repthm}
We also prove a related theorem that connects the failure of the HOD hypothesis
to definable infinitary partition properties, choiceless large cardinals,
and sharps.

A cardinal \(\lambda\) is \textit{J\'onsson} if for all \(f : [\lambda]^{<\omega}\to \lambda\),
there is a proper subset \(A\) of \(\lambda\) of cardinality \(\lambda\)
such that \(f(s) = s\) for all \(s\in [A]^{<\omega}\).
A cardinal \(\lambda\) is \textit{constructibly J\'onsson}
if this holds for all constructible functions \(f : [\lambda]^{<\omega}\to \lambda\),
there is a proper subset \(A\) of \(\lambda\) of cardinality \(\lambda\)
such that \(f(s) = s\) for all \(s\in [A]^{<\omega}\).
\begin{prp*}
    The following are equivalent:
    \begin{itemize}
    \item \(0^\#\) exists.
    \item Every uncountable cardinal is constructibly J\'onsson.
    \item Some uncountable cardinal is constructibly J\'onsson.
    \end{itemize}
\end{prp*}

A cardinal \(\lambda\) is \textit{\(\omega\)-J\'onsson} if for all 
\(f : [\lambda]^{\omega}\to \lambda\),
there is a proper subset \(A\) of \(\lambda\) of cardinality \(\lambda\)
such that \(f(s) = s\) for all \(s\in [A]^{\omega}\).
Recall Kunen's theorem that there is no elementary embedding
\(j : V\to V\). Kunen proved his theorem by showing
that if \(j(\lambda) = \lambda\), then \(\lambda\) is \(\omega\)-J\'onsson.
He then cites the following combinatorial theorem:
\begin{thm*}[Erd\"os-Hajnal]
    There are no \(\omega\)-J\'onsson cardinals.
\end{thm*}
A cardinal \(\lambda\) is \textit{definably \(\omega\)-J\'onsson}
if for all ordinal definable functions 
\(f : [\lambda]^{\omega}\to \lambda\),
there is a proper subset \(A\) of \(\lambda\) of cardinality \(\lambda\)
such that \(f(s) = s\) for all \(s\in [A]^{\omega}\).
\begin{repthm}{thm:jonsson}
    Assume \(\kappa\) is strongly compact.
    Then the following are equivalent:
    \begin{itemize}
        \item The \HOD\ hypothesis fails.
        \item All sufficiently large 
    regular cardinals are definably \(\omega\)-J\'onsson.
        \item Some regular cardinal above \(\kappa\) is definably \(\omega\)-J\'onsson.
    \end{itemize}
\end{repthm}

In \cref{section:hod_rigid}, we use techniques drawn from the
proof of the meain theorem of
\cite{UniqueEmbeddings} to answer a decade-old question of Woodin (implicit in \cite{Woodin}).
Recall Kunen's famous theorem relating the rigidity of \(L\) to \(0^\#\):
\begin{thm*}[Kunen]
    If \(0^\#\) does not exist, then there is no nontrivial elementary embedding from \(L\) to \(L\).
\end{thm*}
Here we prove an analog for \HOD, replacing the assumption that \(0^\#\)
does not exist with the \HOD\ hypothesis.
\begin{repthm}{thm:hod_rigid} 
    Assume there is an \(\omega_1\)-strongly compact cardinal.
    If the \HOD\ hypothesis holds, then there is no nontrivial elementary embedding from \HOD\ to \HOD.
\end{repthm}

\cref{section:ad} turns to the question of the structure
of strongly compact cardinals
in the HOD of a model of determinacy. In this context, something close
to the failure of the HOD hypothesis holds, and yet we will show that
HOD still has certain local covering properties in regions of strong compactness. 
One consequence of our analysis is the following theorem
on the equivalence of strong compactness and \(\omega_1\)-strong compactness in \(\HOD\) 
assuming Woodin's \textit{\(\HOD\) ultrafilter conjecture}: 
under AD + \(V = L(P(\mathbb R))\), every \(\omega_1\)-complete
ultrafilter of \(\HOD\) generates an \(\omega_1\)-complete filter in \(V\).
(The HOD ultrafilter conjecture is true in \(L(\mathbb R)\) and
more generally in all inner models
of determinacy amenable to the techniques of contemporary inner model theory.)
\begin{repthm}{thm:hod_uf}
    Assume \(\AD^+\), \(V = L(P(\mathbb R))\), and the \(\HOD\)-ultrafilter conjecture.
    Suppose \(\kappa < \Theta\) is a regular cardinal and \(\delta\geq \kappa\) is a
    \(\HOD\)-regular ordinal. Then \(\kappa\) is \(\delta\)-strongly compact in \(\HOD\)
    if and only if \(\kappa\) is \((\omega_1,\delta)\)-strongly compact in \(\HOD\).
\end{repthm}
This should be compared with a theorem of \cite{UA} stating
that under the Ultrapower Axiom, for any regular cardinal \(\delta\), 
the least \((\omega_1,\delta)\)-strongly
compact cardinal is \(\delta\)-strongly compact. It is natural to conjecture
that assuming AD, if \(V = L(P(\mathbb R))\), then \(\HOD\) is a model of the Ultrapower Axiom.

The final section (\cref{section:wem}) explores the absoluteness of large cardinals to \(\HOD\)
in the context of a strongly compact cardinal.
Woodin \cite{Woodin} showed:
\begin{thm*}[Woodin]
    Suppose \(\kappa\) is extendible
    and the \HOD\ hypothesis holds. Then \(\kappa\) is extendible in \(\HOD\).
\end{thm*}
By results of Cheng-Friedman-Hamkins \cite{chengfriedmanhamkins},
the least supercompact cardinal need not be weakly compact in \(\HOD\).
Here we show that nevertheless, \(\HOD\) is very close to an inner model with a supercompact
cardinal.
\begin{repthm}{thm:hod_wem}
    Suppose \(\kappa\) is supercompact and the \HOD\ hypothesis holds.
    Then \(\kappa\) is supercompact in an inner model \(N\) of \ZFC\ 
    such that \(\HOD\subseteq N\) and \(\HOD^{<\kappa}\cap N\subseteq \HOD\).
\end{repthm}
As a corollary, we show:
\begin{repthm}{thm:hod_ua}
    Suppose \(\kappa\) is supercompact, the \HOD\ hypothesis holds,
    and \HOD\ satisfies the Ultrapower Axiom.
    Then \(\kappa\) is supercompact in \(\HOD\).
\end{repthm}
\section{On the HOD dichotomy}\label{section:hod_dichotomy}
\subsection{Strongly measurable cardinals}
For any regular cardinal \(\gamma\) and any ordinal \(\delta\) of cofinality greater than \(\gamma\), 
\(\mathcal C_\delta\) denotes
the closed unbounded filter on \(\delta\),
\(S^\delta_\gamma\) denotes the set of ordinals less than \(\delta\) of cofinality
\(\gamma\), and \(\mathcal C_{\delta,\gamma}\) denotes the filter generated by
\(\mathcal C_\delta \cup \{S^\delta_\gamma\}\), or equivalently the filter
generated by the \(\gamma\)-closed unbounded subsets of \(\delta\).
Similarly, \(S^\delta_{<\gamma} = \bigcup_{\lambda < \gamma} S^\delta_{\lambda}\)
and \(\mathcal C_{\delta,{<}\gamma}\) denotes the filter 
generated by \(\mathcal C_\delta\cup \{S^\delta_{<\gamma}\}\).

A filter \(F\) is \textit{\(\gamma\)-saturated}
if any collection of \(F\)-almost disjoint \(F\)-positive sets has
cardinality less than \(\gamma\). The saturation of
a filter is a measure of how close the filter is to being an ultrafilter.
\begin{lma}[Ulam]\label{lma:ulam}
    Suppose \(F\) is a \((2^{\gamma})^+\)-complete \(\gamma\)-weakly saturated filter.
    Then \(F\) is the intersection of fewer than \(\gamma\)-many ultrafilters.\qed
\end{lma}
Note that if a \(\gamma\)-complete filter 
\(F\) is the intersection of fewer
than \(\gamma\)-many ultrafilters,
then the underlying set of \(F\) can be
partitioned into atoms of \(F\).

If \(\nu < \delta\) are regular cardinals, \(\delta\) is \textit{\(\nu\)-strongly measurable}
in an inner model \(N\) if \(\mathcal C_{\delta,\nu}\cap N\) belongs to \(N\) and is
\(\gamma\)-weakly saturated
in \(N\) for some \(N\)-cardinal \(\gamma\) such that
\((2^\gamma)^N < \delta\).
\begin{lma}
    Suppose \(N\) is an inner model of \ZFC,
    \(\nu < \delta\) are regular cardinals, and \(\eta\) is 
    the least \(N\)-cardinal such that \((2^\eta)^N \geq \delta\). 
    Then the following are equivalent:
    \begin{itemize}
        \item \(\delta\) is \(\nu\)-strongly measurable in \(N\).
        \item For some \(\gamma < \eta\),
        \(\mathcal C_{\delta,\nu}\cap N\) can be written as an intersection 
        of fewer than \(\gamma\)-many ultrafilters of \(N\).
        \item For some \(\gamma < \eta\),
        \(\delta\) can be partitioned into \(\gamma\)-many \(\mathcal C_{\delta,\nu}\)-positive
        sets \(\langle S_\alpha : \alpha < \nu\rangle\in N\) such that 
        \((\mathcal C_{\delta,\nu}\restriction S_\alpha)\cap N\) is an ultrafilter in \(N\).\qed
    \end{itemize}
\end{lma}
\subsection{The HOD dichotomy}
Suppose 
\(M\) is an inner model,
\(\lambda\) is a cardinal,
and \(\lambda'\) is an \(M\)-cardinal.
An inner model \(M\) has the \textit{\((\lambda,\lambda')\)-cover property} if for every set \(\sigma\subseteq M\)
of cardinality less than \(\lambda\), there is a set \(\tau\in M\) of \(M\)-cardinality less than \(\lambda'\)
such that \(\sigma\subseteq \tau\). We refer to the \((\lambda,\lambda)\)-cover property as
the \textit{\(\lambda\)-cover property}.
\begin{prp}\label{prp:hod_dichotomy1}
    Suppose \(\kappa\) is strongly compact. Then exactly one of the following holds:
    \begin{enumerate}[(1)]
        \item All sufficiently large regular cardinals are measurable in \(\HOD\).
        \item \(\HOD\) has the \(\kappa\)-cover property.
    \end{enumerate}
\end{prp}
\begin{proof}
    There are two cases.
    \begin{case} There is a \(\HOD\)-cardinal \(\eta\)
        such that for all regular cardinals \(\delta\),
        \(\mathcal C_{\delta,{<}\kappa}\cap \HOD\) is \(\eta\)-saturated in \(\HOD\).
    \end{case}
    In this case, \cref{lma:ulam} implies that all regular cardinals
    \(\delta > (2^\eta)^\HOD\) are measurable in \(\HOD\).
    \begin{case} For every \(\HOD\)-cardinal \(\eta\),
        there is a regular cardinal \(\delta\) such that
        \(\mathcal C_{\delta,{<}\kappa}\cap \HOD\) is not \(\eta\)-saturated in \(\HOD\).
    \end{case}
    We claim that \(\HOD\) has the \(\kappa\)-cover property.
    For this, it suffices to show that for all \(\eta\), 
    there is a \(\kappa\)-complete
    fine ultrafilter \(\mathcal U\) on \(P_\kappa(\eta)\)
    such that \(\HOD\cap P_\kappa(\eta)\in \mathcal U\).
    Then for each \(\sigma\subseteq\eta\) with \(|\sigma| < \kappa\), the set
    \(\{\tau\in P_\kappa(\eta) : \sigma\subseteq \tau\}\) belongs to \(\mathcal U\)
    by fineness and \(\kappa\)-completeness. Thus this set meets \(\HOD\cap P_\kappa(\eta)\),
    which yields a set \(\tau\in \HOD\) such that \(|\tau|^\HOD < \kappa\)
    and \(\sigma\subseteq \tau\).

    Fix an ordinal \(\eta\). Let \(\delta\) be a regular cardinal
    such that \(\mathcal C_{\delta,{<}\kappa}\cap \HOD\) is not \(\eta\)-saturated in \(\HOD\).
    Let \(\mathcal S = \{S_\alpha : \alpha < \eta\}\) be an ordinal definable partition of
    \(S^\delta_{<\kappa}\) into stationary sets.

    Since \(\kappa\) is strongly compact, there is an
    elementary embedding \(j : V\to M\) with critical point \(\kappa\) 
    such that \(M\) has
    the \((\delta^+, j(\kappa))\)-cover property.
    Let \(\delta_* = \sup j[\delta]\). This implies that \(\cf^M(\delta_*) < j(\kappa)\).
    Fix a closed unbounded set \(C\subseteq \delta_*\) such that
    \(C\in M\) and \(|C|^M < j(\kappa)\). 
    
    Let \(\mathcal T = \langle T_\beta : \beta < j(\eta)\rangle = j(\mathcal S)\).
    Let \[\sigma = \{\beta < j(\eta) : T_\beta\cap \delta_*\text{ is stationary in }M\}\]
    Then \(\sigma\) is ordinal definable in \(M\) since \(\mathcal T\) is.
    Notice that for all \(\alpha < \eta\), 
    \(j[S_\alpha]\) is a stationary subset of \(\delta_*\): \(j\) is continuous on \(S^\delta_{<\kappa}\),
    and so there is a continuous increasing cofinal function \(f : \delta\to \delta_*\)
    such that \(j[S_\alpha] = f[S_\alpha]\).
    Since
    \(j[S_\alpha]\subseteq T_{j(\alpha)}\), \(j(\alpha)\in \sigma\).
    On the other hand, for each \(\beta\in \sigma\), 
    let \(f(\beta) = \min (T_\beta\cap C)\). 
    Then \(f : \sigma\to C\) is an injection, so \(|\sigma|^M < j(\kappa)\).
    In other words, \(\sigma\in j(P_\kappa(\eta))\).
    
    Let \(\mathcal U\) be the ultrafilter on \(P_\kappa(\eta)\) derived
    from \(j\) using \(\sigma\). Then \(\HOD\cap P_\kappa(\eta)\in \mathcal U\)
    since \(\sigma\) is ordinal definable in \(M\). Moreover \(\mathcal U\) is fine
    since \(j[\eta]\subseteq \sigma\). Finally \(\mathcal U\) is 
    \(\kappa\)-complete since \(\crit(j) = \kappa\).

    This finishes the proof that \(\HOD\) has the \(\kappa\)-cover property.
\end{proof}

The proof of \cref{prp:hod_dichotomy1} shows that if HOD does not cover \(V\), then
sufficiently large
regular cardinals are measurable in \(\HOD\) in a strong sense.
\begin{prp}\label{prp:hod_dichotomy2}
    Suppose \(\kappa\) is strongly compact. Then exactly one of the following holds:
    \begin{enumerate}[(1)]
        \item For some \(\gamma\), for all ordinals \(\delta\) with \(\cf(\delta) \geq \gamma\),
        \(\mathcal C_{\delta,<\kappa}\cap \HOD\) is the intersection
        of fewer than \(\gamma\)-many ultrafilters in \(\HOD\).
        \item \(\HOD\) has the \(\kappa\)-cover property.\qed
    \end{enumerate}
\end{prp}
For any \(\OD\) set \(A\subseteq\OD\), let \(|A|^\OD\) denote the minimum ordertype of an ordinal
definable wellorder of \(A\). For any set \(X\), let \(\beta(X)\) denote the set of ultrafilters on \(X\).
\begin{prp}
    Suppose \(\kappa\) is strongly compact and \(\HOD\) has the \(\kappa\)-cover property.
    Then for any cardinal \(\lambda \geq \kappa\), 
    \(\HOD\) has the \(({\leq}\lambda,\theta)\)-cover property
    where \(\theta = |P^\OD(\beta(\lambda))|^\OD\).
    \begin{proof}
        We may assume that \(\lambda\) is a regular cardinal.
        Let \(\mathcal U\) be a \(\kappa\)-complete fine ultrafilter on \(P_\kappa(\lambda)\).
        Note that \(\sigma\) is covered by a set of size less than \(j_\mathcal U(\kappa)\) in
        \(M_\mathcal U\). Since \(\HOD^{M_\mathcal U}\) has the \(j_\mathcal U(\kappa)\)-cover property
        in \(M_\mathcal U\), \(\sigma\) is in fact covered by a set of size less than \(j_\mathcal U(\kappa)\)
        in \(\HOD^{M_\mathcal U}\). But \(\HOD^{M_\mathcal U}\subseteq \HOD_\mathcal U\), and
        so \(\sigma\) is covered by a set of size less than \((2^\lambda)^+\)
        in \(\HOD_\mathcal U\). But \(\HOD_\mathcal U\) is a 
        \(\theta\)-cc extension of \(\HOD\) and \(\theta\geq (2^\lambda)^+\), 
        so \(\sigma\) is covered by a set of size less than
        \(\theta\) in \(\HOD\).
    \end{proof}
\end{prp}
\begin{cor}
    If \(\kappa\) is strongly compact, either all sufficiently large regular cardinals
    are measurable in \(\HOD\) or \(\HOD\) has the \(\lambda\)-cover property
    for all strong limit cardinals \(\lambda\geq \kappa\).\qed
\end{cor}
In particular, for all strong limit singular cardinals \(\lambda \geq \kappa\), 
\(\lambda\) is singular in \(\HOD\) and \(\lambda^{+\HOD} = \lambda^+\).
But in fact one can prove this for arbitrary singular cardinals above \(\kappa\).
The proof is more general in two ways. First, it uses a weaker large cardinal hypothesis:
a cardinal \(\kappa\) is \textit{\(\omega_1\)-strongly compact} if every \(\kappa\)-complete
filter extends to a countably complete ultrafilter. 

Second, it applies to a broader class of models than just \(\HOD\):
an inner model \(N\) of ZFC is \textit{\(\omega\)-club amenable} if
for every ordinal \(\delta\) of uncountable cofinality, \(\mathcal C_{\delta,\omega}\cap N\in N\).
The proof of the following
lemma is similar to that of \cref{prp:hod_dichotomy1}:
\begin{lma}\label{lma:club_amenable_cover}
    Suppose \(N\) is \(\omega\)-club amenable and 
    \(\kappa\) is \(\omega_1\)-strongly compact. Then one of the following holds:
    \begin{enumerate}[(1)]
        \item All sufficiently large regular cardinals are measurable in \(N\).
        \item \(N\) has the \((\omega_1,\kappa)\)-cover property.
    \end{enumerate}
    \begin{proof}The proof splits into cases as in \cref{prp:hod_dichotomy1}
        \begin{case} There is an \(N\)-cardinal \(\eta\)
            such that for all regular cardinals \(\delta\),
            \(\mathcal C_{\delta,\omega}\cap N\) is \(\eta\)-saturated in \(N\).
        \end{case}
        Then for all sufficiently large regular cardinals
        \(\delta\), \(\delta\) is measurable in \(N\) by 
        \cref{lma:ulam}, applied in \(N\) to \(\mathcal C_{\delta, \omega}\cap N\) in \(N\).
        \begin{case} For every \(N\)-cardinal \(\eta\),
            there is a regular cardinal \(\delta\) such that
            \(\mathcal C_{\delta,\omega}\cap N\) is not \(\eta\)-saturated in \(N\).
        \end{case}
        In this case, one shows that for all \(\lambda \geq \kappa\), there is a 
        countably complete fine ultrafilter \(\mathcal U\) on \(P_\kappa(\lambda)\)
        such that \(N\cap P_\kappa(\lambda)\in \mathcal U\). It follows that
        every countable set \(\sigma\subseteq \lambda\) belongs 
        to some \(\tau\in N\cap P_\kappa(\lambda)\).
    \end{proof}
\end{lma}

\begin{thm}\label{thm:cf_card}
    Suppose \(\nu\) is a regular cardinal and \(N\) 
    is an \(\omega\)-club amenable inner model with the \((\omega_1,\nu)\)-cover property. 
    Then for all \(N\)-regular \(\delta \geq \nu\), 
    \(\cf(\delta) = |\delta|\).
\end{thm}
\begin{proof}
    Let \(A = (S^\delta_{<\nu})^N\). The \((\omega_1,\nu)\)-cover property of \(N\) implies
    that \(S^\delta_\omega\subseteq A\), which is all we will use.
    Fix \(\langle c_\xi : \xi\in A\rangle\in N\) such that
    \(c_\xi\) is a closed cofinal subset of \(\xi\) of ordertype less than \(\nu\). 
    For each \(\alpha < \delta\), let \(\beta_\alpha\) denote the least ordinal
    such that for a stationary set of \(\xi\in S^\delta_\omega\),
    \(c_\xi\cap [\alpha,\beta_\alpha)\neq \emptyset\). 
    One can prove that \(\beta_\alpha\) exists for all \(\alpha < \delta\)
    by applying Fodor's lemma to the function \(f(\xi) = \min (c_\xi\setminus \alpha)\).
    Define a continuous increasing sequence \(\langle \epsilon_\alpha : \alpha < \delta\rangle\)
    by setting \(\epsilon_{\alpha+1} = \beta_{\epsilon_\alpha}\), taking suprema at limit steps;
    these suprema are always below \(\delta\) because the sequence belongs to \(N\)
    and \(\delta\) is regular in \(N\). 
   
    For each \(\xi \in A\), let  
    \[\sigma_\xi = \{\alpha < \delta : c_\xi\cap [\epsilon_\alpha,\epsilon_{\alpha+1})\neq \emptyset\}\]
    Note that \(|\sigma_\xi| < \nu\) since \(|c_\xi| < \nu\).
    Moreover for all \(\alpha\), the set 
    \(S_\alpha = \{\xi \in A : \alpha\in \sigma_\xi\}\) is stationary.
    Let \(C\subseteq \delta\) be a closed cofinal set of ordertype \(\cf(\delta)\).
    For any \(\alpha < \delta\), there is some \(\xi\in S_\alpha \cap C\),
    which means that \(\alpha\in \sigma_\xi\).
    This implies that \(\delta = \bigcup_{\xi\in C} \sigma_\xi\).
    Thus \(|\delta| = |C|\cdot \sup_{\xi\in C} |\sigma_\xi|\).
    If \(|C| < \nu\), then since \(\nu\) is regular, \(\sup_{\xi\in C} |\sigma_\xi| < \nu\) and hence
    \(|\delta| < \nu\), contradicting that \(\delta \geq \nu\). Therefore
    \(|C|\geq \nu\). 
    Therefore \(|\delta| = |C|\cdot \sup_{\xi\in C} |\sigma_\xi| = |C|\cdot \nu = |C| = \cf(\delta)\).
\end{proof}

\begin{cor}\label{cor:club_amenable_weak_cover}
    Suppose \(\kappa\) is \(\omega_1\)-strongly compact and \(N\) is
    an \(\omega\)-club amenable inner model. Either all 
    sufficiently large regular cardinals are measurable in \(N\)
    or every singular cardinal \(\lambda\) greater than \(\kappa\) is singular in \(N\)
    and \(\lambda^{+N} = \lambda^+\).\qed
\end{cor}
We now state some reformulations of the failure of the HOD hypothesis.
\begin{thm}\label{thm:stationary_partitions}
    If \(\kappa\) is strongly compact, the following are equivalent
    to the failure of the \HOD\ hypothesis:
    \begin{enumerate}[(1)]
        \item\label{item:no_partition} There is a regular cardinal \(\delta\geq \kappa\)
        with a stationary subset \(S\subseteq S^\delta_{<\kappa}\)
        that admits no ordinal definable partition into \(\delta\)-many stationary sets.
        \item There is an \(\omega\)-strongly measurable cardinal above \(\kappa\).
        \item For some regular \(\gamma < \kappa\), there is a \(\gamma\)-strongly measurable cardinal 
        above \(\kappa\).
        \item For all regular
        \(\gamma < \kappa\), all sufficiently large regular cardinals
        are \(\gamma\)-strongly measurable.
        \item\label{item:eventual_saturation} For some cardinal \(\lambda\), for all regular \(\gamma < \kappa\)
        and all \(\nu\) with \(\cf(\nu) > \gamma\), 
        \(\mathcal C_{\nu,\gamma}\cap \HOD\) is \(\lambda\)-saturated in \(\HOD\).
    \end{enumerate}
    \begin{proof}
        We first show that \ref{item:no_partition} implies that the \HOD\ hypothesis is false.
        By the proof of \cref{prp:hod_dichotomy1}, if the \HOD\ hypothesis fails,
        then \ref{item:eventual_saturation} holds. Since each item easily implies
        all the previous ones, the theorem easily follows.

        It remains to show that \ref{item:no_partition} implies the failure of the \HOD\ hypothesis,
        or in other words, that the \HOD\ hypothesis implies that
        every stationary subset \(S\subseteq S^\delta_{<\kappa}\) splits into \(\delta\)-many subsets.
        This follows from the proof of the Solovay splitting theorem and the fact that
        \((S^\delta_{<\kappa})^\HOD = S^\delta_{<\kappa}\). 
    \end{proof}
\end{thm}
It is easy to see that if the \(\HOD\) hypothesis fails, then 
every regular cardinal \(\delta\) contains an \(\omega\)-club of ordinals
that are strongly inaccessible in \(\HOD\).
In fact, this alone implies a stronger conclusion.
\begin{thm}\label{thm:omega_club_to_club}
    Suppose \(N\) is an inner model and \(\delta\) is a regular cardinal
    such that \(\Reg^N\cap \delta\) contains an \(\omega\)-club.
    Then \(\Reg^N\cap \delta\) contains a club.
    \begin{proof}
        Assume towards a contradiction that \(\Sing^N\cap \delta\) is stationary.
        Since the cofinality function as computed in 
        \(N\) is regressive, there is an ordinal
        \(\gamma\) that is regular in \(N\)
        such that \((S^\delta_\gamma)^N\) is stationary.
        Let \(\gamma\) be the least such ordinal and let
        \(S = (S^\delta_\gamma)^N\). We claim \(\cf(\gamma) = \omega\).

        Choose \(\langle c_\xi : \xi\in S\rangle\in N\)
        such that \(c_\xi\) is a closed unbounded subset
        of \(\xi\) with ordertype \(\gamma\). Consider the
        set \(T\) of \(\nu\in S^\delta_\omega\) such that
        there is some \(\xi \geq \nu\) in \(S\) 
        such that \(c_\xi\cap \nu\) is cofinal in \(\nu\).
        Then \(T\) is stationary. To see this, fix a closed unbounded
        set \(C\subseteq \delta\), and we will show that \(C\cap T\neq \emptyset\).
        Fix \(\xi\in S\cap \acc(C)\). 
        Since \(\cf(\xi) > \omega\), 
        \(C\cap \acc(c_\gamma)\) is closed unbounded in \(\gamma\), and
        so there is some \(\nu\in C\cap \acc(c_\gamma)\cap S^\delta_\omega\).
        By definition, \(\nu\in T\). So \(C\cap T\neq \emptyset\), as claimed.
        
        Note that
        \(T\subseteq (S^\delta_{<\gamma})^N\), so \((S^\delta_{<\gamma})^N\)
        is stationary. Again applying Fodor's lemma,
        there is an \(N\)-regular ordinal \(\gamma' < \gamma\) 
        such that \((S^\delta_{\gamma'})^N\) is stationary, and this contradicts
        the minimality of \(\gamma\).

        Since \((S^\delta_\gamma)^N\) is a stationary subset of \(S^\delta_\omega\), 
        its intersection with any \(\omega\)-club is stationary.
        This implies that 
        \((S^\delta_\gamma)^N\cap \Reg^N\) is stationary, which is a contradiction.
    \end{proof}
\end{thm}
\begin{cor}
    Assume there is an \(\omega_1\)-strongly compact cardinal \(\kappa\) and \(N\) is
    an \(\omega\)-club amenable model such that \(\lambda^{+N} < \lambda^+\) for some
    singular \(\lambda > \kappa\).
    Then all sufficiently large regular cardinals 
    contain a closed unbounded set of \(N\)-inaccessible cardinals.
    \begin{proof}
        Suppose \(\delta\) is \(\omega\)-strongly measurable in \(N\),
        and we will show that there is an \(\omega\)-club of \(N\)-regular ordinals below \(\delta\).
        Suppose not. Then \(\Sing^N\) is \(\mathcal C_{\delta,\omega}\)-positive.
        It follows that there is a set \(S\in N\) such that \(S\subseteq \Sing^N\)
        and \(\mathcal C_{\delta,\omega}\restriction S\) is a normal ultrafilter in \(N\).
        This is a contradiction since \(\Reg^N\cap \delta\) belongs to any normal ultrafilter
        of \(N\) on \(\delta\).

        Since all sufficiently large regular cardinals are \(\omega\)-strongly measurable in \(N\)
        by \cref{cor:club_amenable_weak_cover}, the desired conclusion follows
        from \cref{thm:omega_club_to_club}.
    \end{proof}
\end{cor}

\begin{cor}
    Assume there is an \(\omega_1\)-strongly compact cardinal and the \(\HOD\) hypothesis fails.
    Then all sufficiently large regular cardinals
    contain a closed unbounded set of \(\HOD\)-inaccessible cardinals.\qed
\end{cor}
\section{Embeddings of HOD}\label{section:uniqueness}
\subsection{Uniqueness of elementary embeddings}
We prove that the HOD Hypothesis is equivalent to a uniqueness property of elementary embeddings
that makes no mention of ordinal definability.
\begin{lma}
    If \(\kappa\) is supercompact, then for all regular \(\delta \geq \kappa\) and all \(\gamma > \delta\), the set 
    \[\{\sup j[\bar \delta] : \bar \delta < \bar \gamma < \delta,\ j:V_{\bar\gamma}\to V_\gamma
    \text{ is elementary, and }j(\bar \delta) = \delta\}\]
    is stationary in \(\delta\).
\end{lma}
\begin{lma}\label{lma:hod_regular}
    Suppose \(\kappa\) is strongly compact and for arbitrarily
    large regular cardinals \(\delta \geq \kappa\),
    the set \(\{\alpha < \delta : \cf^\HOD(\alpha) < \kappa\}\) is stationary. Then
    the \(\HOD\) hypothesis holds.
    \begin{proof}
        This follows from \cref{prp:hod_dichotomy2}.
    \end{proof}
\end{lma}
\begin{defn}
    Suppose \(j_0,j_1 : M\to N\) are elementary embeddings. For any ordinal
    \(\delta\in M\), \(j_0\) and \(j_1\)
    are \textit{similar at \(\delta\)} if \(\sup j_0[\delta] = \sup j_1[\delta]\)
    and \(j_0(\delta) = j_1(\delta)\). For any ordinal \(\delta'\in N\),
    \(j_0\) and \(j_1\) are \textit{similar below \(\delta'\)} if 
    there is some \(\delta \in M\) such that \(j_0(\delta) = j_1(\delta) = \delta'\)
    and \(j_0\) and \(j_1\) are similar at \(\delta\).
\end{defn}

\begin{thm}
    Suppose \(\kappa\) is a supercompact cardinal. Then the following are equivalent:
    \begin{enumerate}[(1)]
        \item The \(\HOD\) hypothesis holds.\label{item:hod_hyp}
        \item 
        For all regular \(\delta\geq \kappa\), for all sufficiently large \(\gamma\),
        if \(j_0,j_1 : V_{\bar \gamma}\to V_\gamma\) are similar below \(\delta\),
        then \(j_0\restriction \bar \delta = j_1\restriction \bar \delta\).\label{item:universal_sub}
    \end{enumerate}
    \begin{proof}
        Assume \ref{item:hod_hyp} holds. Let 
        \(\langle S_\alpha : \alpha < \delta\rangle\) be the least partition of
        \(S^\delta_\omega\) into stationary sets
        in the canonical wellorder of \(\HOD\). For sufficiently large ordinals \(\gamma\),
        \(\langle S_\alpha : \alpha < \delta\rangle\) is definable from \(\delta\)
        in \(V_\gamma\).
        Suppose \(j: V_{\bar \gamma}\to V_\gamma\) and \(j(\bar \delta) = \delta\).
        Then \(\langle S_\alpha : \alpha < \delta\rangle\) is in the range of \(j\) since \(\delta\) is. 
        As a consequence,
        \(j[\bar \delta]\) is the set of \(\alpha < \delta\) such that \(S_\alpha\)
        reflects to \(\sup j[\bar \delta]\). \ref{item:universal_sub} follows immediately.

        Conversely, assume \ref{item:universal_sub}. Fix a regular cardinal \(\delta > \kappa\).
        For each \(\gamma > \delta\), define
        \[S_\gamma = \{\sup j[\bar \delta] : j:V_{\bar\gamma}\to V_\gamma\text{ is elementary and }j(\bar \delta) = \delta\}\]
        Fix an ordinal \(\gamma\) sufficiently large that if \(j_0,j_1 : V_{\bar \gamma}\to V_\gamma\) are elementary embeddings with 
        \(j_0(\bar \delta) = j_1(\bar \delta) = \delta\) and \(\sup j_0[\bar \delta] = \sup j_1[\bar \delta]\),
        then \(j_0\restriction \bar \delta = j_1\restriction \bar \delta\). 
        For \(\xi \in S_\gamma\), let \(c_\xi = j[\bar \delta]\)
        for \(j : V_{\bar \gamma}\to V_\gamma\) such that \(j(\bar \delta) = \delta\) and 
        \(\sup j[\bar\delta] = \xi\).
        Then \(c_\xi\) is ordinal definable, so \(\cf^\HOD(\xi) \leq \ot(c_\xi)\).
        By \cref{lma:hod_regular}, the HOD hypothesis holds. 
    \end{proof}
\end{thm}
As a corollary, one has the following equivalent forms of the local uniqueness of
elementary embeddings at an extendible. We note that the equivalences would
be quite mysterious (and hard to prove) without having the \HOD\ hypothesis to tie them together.
\begin{thm}\label{thm:extendible_uniqueness}
    If \(\kappa\) is extendible, then the following are equivalent.
    \begin{enumerate}[(1)]
        \item The \HOD\ hypothesis holds.
        \item For all regular \(\delta \geq \kappa\), for all sufficiently large \(\alpha\),
        if \(j_0,j_1 : V_\alpha\to M\) are similar at \(\delta\),
        then \(j_0\restriction \delta = j_1\restriction \delta\).
        \item For some regular \(\delta \geq \kappa\), for all sufficiently large \(\alpha\),
        if \(j_0,j_1 : V_\alpha\to V_{\alpha'}\) are similar at \(\delta\),
        then \(j_0\restriction \delta = j_1\restriction \delta\).
        \item Let \(\nu\) be the least supercompact cardinal. 
        For all regular \(\delta\geq \nu\), for all sufficiently large \(\alpha\),
        if \(j_0,j_1 : V_{\bar \alpha}\to V_\alpha\) are similar below \(\delta\),
        then \(j_0\restriction \delta = j_1\restriction \delta\).\qed
    \end{enumerate}
\end{thm}
\subsection{\(\omega\)-J\'onsson cardinals}
Recall that a cardinal \(\lambda\) is \textit{\(\omega\)-J\'onsson}
if for all functions \(f : [\lambda]^\omega\to \lambda\), there is some
\(H\subseteq \lambda\) such that \(\ot(H) = \lambda\) and
\(\ran(f\restriction [H]^\omega)\) is a proper subset of \(\lambda\).
The Erdos-Hajnal theorem states that assuming ZFC there is no such cardinal.
One can ask what happens when one places definability constraints on \(f\).
\begin{defn}
    A cardinal \(\lambda\) is \textit{definably \(\omega\)-J\'onsson} if for all ordinal 
    definable functions
    \(f : [\lambda]^{\omega}\to \lambda\), there is some \(H\subseteq \lambda\)
    such that \(\ot(H) = \lambda\) and
    \(\ran(f\restriction [H]^\omega)\) is a proper subset of \(\lambda\).
\end{defn}

\begin{thm}\label{thm:jonsson}
    Suppose \(\kappa\) is strongly compact. Then the following are equivalent:
    \begin{enumerate}[(1)]
        \item\label{item:hod_hyp_to_Jonsson} The \HOD\ hypothesis fails.
        \item\label{item:jonsson} All sufficiently large
        regular cardinals are definably \(\omega\)-J\'onsson.
        \item\label{item:some_jonsson} Some regular cardinal above \(\kappa\)
        is definably \(\omega\)-J\'onsson.
    \end{enumerate}
    \begin{proof}
        \textit{\ref{item:hod_hyp_to_Jonsson} implies \ref{item:jonsson}:}
        The converse rests on Solovay's published proof of Solovay's lemma \cite{Solovay}.
        Let \(R\) be the class of regular
        cardinals \(\delta\) that are not definably \(\omega\)-J\'onsson.
        Fix \(\delta\in R\), and we will show that
        \(\HOD\) has the \(\kappa\)-cover property below \(\delta\).
        This completes the proof, because assuming \ref{item:jonsson}
        fails, \(R\) is a proper class, and hence \(\HOD\) has the \(\kappa\)-cover property,
        which, by \cref{prp:hod_dichotomy1}, 
        is equivalent to the \(\HOD\) hypothesis given a strongly compact cardinal.

        Let \(f : [\delta]^\omega\to \delta\) be an ordinal
        definable function such that for all \(A\subseteq \delta\)
        such that \(f[[A]^\omega] = \delta\).
        We claim that any \(\omega\)-closed 
        unbounded subset \(C\) of \(\sup j[\delta]\) such that
        \(j(f)[C]\subseteq C\) contains \(j[\delta]\).
        Let \(B = j^{-1}[C]\). Then \(B\) is unbounded below \(\delta\)
        and \(B\) is closed under \(f\), and therefore
        \(B\supseteq f[B] = \delta\). In other words, \(j[\delta]\subseteq C\).

        Working in \(M\), 
        let \(A\) be the intersection of all \(\omega\)-closed unbounded subsets of
        \(\sup j[\delta]\) that are closed under \(j(f)\). Then \(A\in M\)
        and \(j[\delta]\subseteq A\) and \(A\) is ordinal definable from \(j(f)\) in \(M\).
        Therefore \(A\in \HOD^M\). It follows from the proof of \cref{prp:hod_dichotomy1}
        that \(\HOD\) has the \(\kappa\)-cover property below \(\delta\).

        \textit{\ref{item:jonsson} implies \ref{item:some_jonsson}:} Trivial.

        \textit{\ref{item:some_jonsson} implies \ref{item:hod_hyp_to_Jonsson}:}
        Suppose \(\delta\geq \kappa\)
        is regular and definably \(\omega\)-J\'onsson,
        and assume towards a contradiction that the \HOD\ hypothesis holds.
        Applying \cref{thm:stationary_partitions},
        let \(\langle S_\alpha : \alpha < \delta\rangle\in \HOD\) partition 
        \(S^\delta_\omega\) into stationary sets. Let \(f : [\delta]^\omega\to \delta\)
        be defined by \(f(s) = \alpha\) where \(\alpha < \delta\)
        is the unique ordinal such that \(\sup(s)\in S_\alpha\).
        Then for any unbounded set \(T\subseteq \delta\), 
        for each \(\alpha < \delta\), there is some \(s\in [T]^\omega\)
        such that \(\sup s\in S_\alpha\), and therefore
        \(f(s) = \alpha\). It follows that \(f[[T]^\omega]] = \delta\).
    \end{proof}
\end{thm}
The usual characterizations of J\'onsson cardinals in terms of elementary embeddings yields the
following equivalence.
\begin{cor}
    Assume there is a strongly compact cardinal. Then exactly one of the following holds.
    \begin{enumerate}[(1)]
        \item The \HOD\ hypothesis.
        \item For all sufficiently large regular cardinals \(\delta\),
        for all ordinals \(\alpha > \delta\), there is an elementary embedding
        \(j : M \to \HOD\cap V_\alpha\) such that \(\crit(j) < \delta\), \(j(\delta) = \delta\),
        and \(j\) is continuous at ordinals of cofinality \(\omega\).\qed
    \end{enumerate}
\end{cor}
\subsection{The rigidity of HOD}\label{section:hod_rigid}
This section is devoted to a proof of the following theorem:
\begin{thm}\label{thm:hod_rigid} 
    Assume there is an \(\omega_1\)-strongly compact cardinal.
    If the \HOD\ hypothesis holds, then there is no nontrivial elementary embedding from \HOD\ to \HOD.
\end{thm}

We need to prove some preliminary lemmas.

\begin{lma}\label{lma:hod_embedding_proper_class}
    If \(j : \HOD\to \HOD\) is a nontrivial elementary embedding, then \(j\) has a proper class of generators.\qed
\end{lma}

\begin{lma}\label{lma:club_amenable_sch}
    Suppose \(\kappa\) is \(\omega\)-strongly compact and
    \(H\) is an \(\omega\)-club amenable model such that
    arbitrarily large regular cardinals are not measurable in \(H\).
    Then the Singular Cardinals Hypothesis holds in \(H\) above \(\kappa^+ \cdot (\kappa^\omega)^H\).
    \begin{proof}
        By Silver's theorem, 
        it suffices to show that in \(H\), for all singular cardinals
        \(\lambda > 2^{<\kappa}\) with \(\cf(\lambda) = \omega\),
        \(\lambda^\omega = \lambda^+\). In fact,
        we will show that for all \(H\)-regular \(\delta \geq \kappa^+\),
        there is a set \(C\subseteq P_\kappa(\delta)\cap H\)
        in \(H\) such that \(|C|^H = \delta\) and
        every countable \(\sigma\subseteq \delta\)
        is contained in some \(\tau \in H\).
        Then in particular, \(([\delta]^\omega)^H\subseteq
        \bigcup_{\sigma\in C} ([\sigma]^\omega)^H\)
        and hence \((\delta^\omega)^H = (\kappa^\omega)^H\cdot \delta\).

        To define \(C\), we first split
        \(S = (S^\delta_{<\kappa})^H\) into \(\delta\)-many \(\mathcal C_{\delta,\omega}\)-positive
        sets in \(H\). Note that \(S\) contains \(S^\delta_\omega\) by \cref{lma:club_amenable_cover}, 
        and hence \(S\)
        has full measure with respect to \(\mathcal C_{\delta,\omega}\). 
        Now one can use a standard argument to partition 
        \(S\) into \(\delta\)-many \(\mathcal C_{\delta,\omega}\)-positive sets
        in \(H\). Working in \(H\), for each \(\alpha\in S\), fix a club
        \(c_\alpha\subseteq \alpha\) of ordertype less than \(\kappa\). For
        \(\xi < \kappa\), let \(f_\xi(\alpha) = c_\alpha(\xi)\) for any \(\alpha\) such that
        \(\xi < \ot(c_\alpha)\). 

        We claim that
        there is some \(\xi < \kappa\) such that for unboundedly many \(\beta < \delta\), 
        \(f_\xi^{-1}\{\beta\}\) is \(\mathcal C_{\delta,\omega}\)-positive. 
        Otherwise, using that \(C_{\delta,\omega}\) is weakly normal and \(f_\xi\) is regressive, 
        for each \(\xi < \kappa\), there is an ordinal \(\gamma_\xi\) such that
        \(f_\xi(\alpha) < \gamma_\xi\) for an \(\omega\)-club of \(\alpha\in S\). 
        Since \(\cf(\delta) > \kappa\),
        there is then a single \(\omega\)-club of \(\alpha\) 
        such that for all \(\xi < \kappa\), 
        \(f_\xi(\alpha) < \gamma_\xi\) for all appropriate \(\alpha\). Letting
        \(\gamma = \sup_{\xi < \kappa} \gamma_\xi < \delta\), we see that 
        for an \(\omega\)-club of \(\alpha\), for all 
        appropriate \(\xi < \kappa\), \(f_\xi(\alpha) < \gamma\).
        But if \(\alpha > \gamma\) belongs to this club, then the fact that 
        \(c_\alpha(\xi) = f_\xi(\alpha) < \gamma\) for all \(\xi < \ot(c_\alpha)\)
        contradicts that \(c_\alpha\) was chosen to be unbounded in \(\alpha\).

        Now fix \(\xi < \kappa\) such that such that for unboundedly many \(\beta < \delta\), 
        \(f_\xi^{-1}\{\beta\}\) is \(\mathcal C_{\delta,\omega}\)-positive. 
        Since \(H\) is \(\omega\)-club amenable and
        \(f_\xi\in N\), the set \(B\) of such \(\beta\) belongs to \(H\),
        and so since \(\delta\) is regular in \(H\), the ordertype of \(B\) is \(\delta\).
        Finally, let \(S_\beta = f_\xi^{-1}\{\beta\}\).
        Then \(\langle S_\beta : \beta\in B\rangle\) is the desired stationary partition.

        Finally, for each \(\xi < \delta\) of uncountable cofinality, 
        let \(\sigma_\xi\) be the set of \(\beta \in B\) such that
        \(S_\beta\) is \(\mathcal C_{\delta,\xi}\)-positive.
        Then \(\langle \sigma_\xi : \xi < \delta\rangle\) is in \(H\) 
        by \(\omega\)-club amenability, and the argument of \cref{lma:club_amenable_cover}
        shows that every countable subset of \(\delta\) is contained in 
        \(\sigma_\xi\) for some \(\xi < \delta\). Therefore \(C = \{\sigma_\xi : \xi < \delta\}\)
        is as desired.
    \end{proof}
\end{lma}

\begin{lma}\label{lma:sch_extender}
    Suppose \(\delta\) is a cardinal, \(N\) and \(H\) are
    models with the \((\omega_1,\delta)\)-cover property,
    and \((2^\lambda)^H = \lambda^{+H}\) for all sufficiently large 
    strong limit cardinals of countable cofinality.
    Then any elementary embedding \(k : N \to H\) with
    \(\crit(k) \geq \delta\) is an extender embedding.
    \begin{proof}
        Note that \(k\) is continuous at ordinals of cofinality
        \(\omega\), since these have \(N\)-cofinality less than \(\delta\). 
        We will show that there cannot exist a singular strong cardinal \(\lambda\) with the
        following properties:
        \begin{itemize}
            \item \(\cf(\lambda) = \omega\).
            \item \(\lambda\) is a limit of generators of \(k\).
            \item \(k(\lambda) = \lambda\).
            \item \((2^\lambda)^H = \lambda^{+H}\).
        \end{itemize}
        We first point out that \(\lambda^{+H} = \lambda^{+N} = \lambda^+\).
        Using the \((\omega_1,\delta)\)-cover property, 
        \(\lambda^+ \leq |P_\delta(\lambda)\cap H| \leq (2^\lambda)^H 
        =\lambda^{+H}\).
        The argument for \(N\) is similar, using that 
        \((2^\lambda)^N = \lambda^{+N}\) by elementarity.

        Let \(\sigma\subseteq \lambda\) be a countable cofinal set,
        and let \(\tau\in H\) be a cover of \(\sigma\) of size less than \(\delta\). 
        Let \(U\) be the \(N\)-ultrafilter on \(P_\delta(\lambda)\cap N\)
        derived from \(k\) using \(\tau\). Let \(j_U :N\to P\) be the ultrapower embedding and
        \(i : P\to H\) be the factor embedding. Then \(\tau = i(\id_U)\in \ran(i)\), and
        so since \(|\tau| < \delta\) and \(\crit(i) < \delta\), 
        \(\tau\subseteq \ran(i)\). As a consequence, \(\lambda\) is a limit of generators of \(j_U\),
        which implies \(\lambda_U\geq \lambda\).
        But \((2^\lambda)^H = \lambda^{+H}\), and so by elementarity
        \((2^\lambda)^N = \lambda^{+N}\). Thus 
        \(\lambda_U\leq |P_\delta(\lambda)\cap N| \leq (2^\lambda)^N \leq \lambda^{+N}\).
        Since \(\lambda \leq \lambda_U\leq\lambda^{+N}\),
        we must have \(j_U(\lambda) > \lambda\) or \(j_U(\lambda^{+N}) > \lambda^{+N}\).
        This implies \(k(\lambda) > \lambda\) or \(k(\lambda^{+N})> \lambda^{+N}\),
        but contradicting either that \(k(\lambda) = \lambda\) (by assumption) or that 
        \(k(\lambda^{+N}) = \lambda^{+H} = \lambda^{+N}\)
        (by the previous paragraph).

        Since there is no such \(\lambda\), the class of limits of generators of \(k\)
        does not intersect the \(\omega\)-closed unbounded class of singular strong limit 
        fixed points of \(k\) of countable cofinality satisfying \((2^\lambda)^H = \lambda^{+H}\). 
        It follows that the class of limits of generators of \(k\) not closed unbounded,
        which means that the class of generators of \(k\) is bounded, or in other words,
        \(k\) is an extender embedding.
    \end{proof}
\end{lma}

\begin{lma}\label{lma:extender_cover}
    Suppose \(i : H\to N\) is an extender embedding of length \(\lambda\)
    where \(\lambda\) is an \(H\)-cardinal of uncountable cofinality. 
    If \(H\) has the \((\omega_1,\lambda)\)-cover property,
    then so does \(N\).
\end{lma}

\begin{proof}[Proof of \cref{thm:hod_rigid}]
    Suppose \(j : \HOD\to \HOD\) is an elementary embedding and
    assume towards a contradiction that it is nontrivial. 
    Let \(\kappa\) be the least \(\omega_1\)-strongly
    compact cardinal. By \cref{lma:club_amenable_cover}, \(\HOD\) has the
    \((\omega_1,\kappa)\)-cover property. Let 
    \(i : \HOD\to N\) be given by the extender of length \(\kappa\) derived from \(j\),
    and let \(k : N\to \HOD\) be the factor embedding, so \(\crit(k)\geq \kappa\).
    
    By \cref{lma:extender_cover}, \(N\) has the 
    \((\omega_1,\kappa)\)-cover property.
    \cref{lma:club_amenable_sch} implies that 
    \(\HOD\) satisfies the cardinal arithmetic
    condition of \cref{lma:sch_extender}, and so \(k\) is an extender embedding.
    Thus \(j = k\circ i\) is an extender embedding, 
    contradicting \cref{lma:hod_embedding_proper_class}.
\end{proof}
\section{Large cardinals in HOD}
\subsection{HOD under AD}\label{section:ad}
Recall that a cardinal \(\kappa\) is \textit{\((\nu,\delta)\)-strongly compact}
if there is an elementary embedding \(j : V\to M\) such that \(\crit(j)\geq \nu\)
and \(M\) has the \((\delta^+,j(\kappa))\)-cover property.
An ultrafilter \(U\) 
\textit{witnesses that \(\kappa\) is \((\nu,\delta)\)-strongly compact} if \(\crit(j_U) \geq \nu\)
and \(M_U\) has the \((\delta^+,j(\kappa))\)-cover property. If \(N\) is an inner model and \(U\)
is an ultrafilter, then \textit{\(U\) witnesses that \(\kappa\) is \(\delta\)-strongly compact in \(N\)} 
if \(U\cap N\in N\)
and \(U\cap N\) witnesses that \(\kappa\) is \(\delta\)-strongly compact in \(N\).

We say an inner model \(N\) has the \textit{\(\kappa\)-cover property below \(\delta\)}
if \(P_\kappa(\delta)\cap N\) is cofinal in \(P_\kappa(\delta)\);
\(N\) has the \textit{\(\kappa\)-approximation property below \(\delta\)}
if \(N\) contains every \(A\subseteq \delta\) such that \(A\cap \sigma\in N\) for all \(\sigma\in P_\kappa(\delta)\).
\begin{thm}
    Assume \(\AD^+\) and \(V = L(P(\mathbb R))\).
    Suppose \(\kappa < \Theta\) is a regular cardinal and \(\delta\geq \kappa\) is a
    \(\HOD\)-regular ordinal. Then the following are equivalent:
    \begin{enumerate}[(1)]
        \item \((S^\delta_{<\kappa})^\HOD\) is stationary.\label{item:stat}
        \item \(\HOD\) has the \(\kappa\)-cover property below \(\delta\).\label{item:cover}
        \item \(\HOD\) has the \(\kappa\)-cover  
        and \(\kappa\)-approximation properties below \(\delta\).\label{item:appx}
        \item \(\mathcal C_{\delta,\omega}\) witnesses 
        that \(\kappa\) is \(\delta\)-strongly compact in \(\HOD\).\label{item:club}
        \item Some countably complete ultrafilter witnesses 
        \(\kappa\) is \((\omega_1,\delta)\)-strongly compact in \(\HOD\).\label{item:bagmag}
        \item \(\HOD\) has the \((\omega_1,\kappa)\)-cover property below \(\delta\).\label{item:cover2}
        \item \(S^\delta_\omega\subseteq (S^\delta_{<\kappa})^\HOD\).\label{item:sub}
    \end{enumerate}
    \begin{proof}
        \textit{\ref{item:stat} implies \ref{item:cover}} Let \(T = (S^\delta_{<\kappa})^\HOD\). We claim
        that \(S^\delta_\omega \cap T\) is stationary.
        Fix a sequence \(\langle c_\alpha : \alpha\in T\rangle \in \HOD\)
        such that \(c_\alpha\) is a closed cofinal subset of \(\alpha\) of ordertype
        \(\cf^\HOD(\alpha)\). 
        Let \[S = \{\beta \in S^\delta_\omega : \exists \alpha\in T\, \sup (c_\alpha\cap \beta) = \beta\}\]
        Suppose \(C\) is closed unbounded, and we will show \(C\cap S\neq \emptyset\).
        Since \(T\) is stationary,
        there is some \(\beta\) in \(\acc(C)\) such that 
        \(\sup (c_\alpha\cap \beta) = \beta\). We claim that
        the least such \(\beta\) has countable cofinality.
        If not, then \(\beta\) has uncountable cofinality, 
        so \(C\cap c_\alpha\cap\beta\) is closed unbounded.
        But fixing any \(\beta'\in \acc(C\cap c_\alpha \cap \beta)\),
        we have \(\beta'\in \acc(C)\) and \(\sup(c_\alpha\cap \beta') = \beta'\),
        contrary to the minimality of \(\beta\).

        Next, we construct a sequence
        \(\langle \sigma_\xi : \xi\in T\rangle\)
        such that \(\sigma_\xi\in P_\kappa(\xi)\) 
        and for all \(\alpha\), for \(\mathcal C_{\delta,\omega}\)-almost
        all \(\xi\), \(\alpha\in \sigma_\xi\). This proceeds as in \cref{thm:cf_card},
        using however that \(\mathcal C_{\delta,\omega}\) is an ultrafilter.
        As in \cref{thm:cf_card},
        one can use this sequence to show \(\cf(\delta) = |\delta|\). Let \(C\subseteq \delta\) be a closed cofinal
        set of ordertype \(\cf(\delta)\).
        Then \(\delta = \bigcup_{\xi\in C}\sigma_\xi\). Therefore
        \(|\delta| = |C|\cdot\sup_{\xi \in C} \sigma_\xi\), so as in 
        \cref{thm:cf_card}, \(|\delta| = |C| = \cf(\delta)\).


        Finally, suppose \(\sigma\in P_\kappa(\delta)\). 
        The set \(\{\xi \in T : \sigma\subseteq \sigma_\xi\}\)
        is the intersection of fewer than \(\kappa\)-many sets in \(\mathcal C_{\delta,\omega}\),
        and so it belongs to \(\mathcal C_{\delta,\omega}\).
        It follows that there is some \(\xi\in T\) such that \(\sigma\subseteq \sigma_\xi\).
        This shows that \(\HOD\) has the \(\kappa\)-cover property below \(\delta\).

        \textit{\ref{item:cover} implies \ref{item:appx}.} 
        In \(\HOD\), 
        let \(\langle X_\alpha : \alpha < \delta\rangle\) be a \(\kappa\)-independent
        family of subsets of some set \(S\). Suppose
        \(A\subseteq \delta\) and \(A\cap \sigma\in \HOD\) for all \(\sigma\in P_\kappa(\delta)\cap \HOD\).
        For \(\alpha < \delta\), let \[Y_\alpha =\begin{cases} X_\alpha&\text{if }\alpha\in A\\
            S\setminus X_\alpha&\text{otherwise}\end{cases}\]
        Suppose \(\sigma\subseteq \delta\) are disjoint. We will show that
        \(\bigcap_{\alpha\in \sigma}Y_\alpha\neq \emptyset\).
        First, let \(\tau\in \HOD\) cover \(\sigma\). Then
        by our assumption on \(A\), \(\tau\cap A\in \HOD\).
        It follows that \(\langle Y_\alpha : \alpha\in \tau\rangle\in \HOD\),
        and therefore since  \(\langle X_\alpha : \alpha < \delta\rangle\) is \(\kappa\)-independent
        in \(\HOD\), \(\bigcap_{\alpha\in \tau}Y_\alpha\neq \emptyset\).
        Since \(\sigma\subseteq \tau\), \(\bigcap_{\alpha\in \sigma}Y_\alpha\neq \emptyset\).
        Let \(F\) be the filter generated by \(\{Y_\alpha : \alpha < \delta\}\).
        Then \(F\) is a \(\kappa\)-complete filter on a wellorderable set, 
        and so \(F\) extends to a countably complete
        ultrafilter \(U\). Applying AD, \(U\cap \HOD\in \HOD\), and therefore
        \(A = \{\alpha < \delta : X_\alpha\in U\}\) belongs to \(\HOD\) as well. 

        \textit{\ref{item:appx} implies \ref{item:club}.} 
        \((S^\delta_{<\kappa})^\HOD\in \mathcal C_{\delta,\omega}\) and
        \(C_{\delta,\omega}\) is \(\kappa\)-complete since \(\cf(\delta)\geq \kappa\).
        Therefore \(\mathcal C_{\delta,\omega}\) is a \(\kappa\)-complete,
        \((\kappa,\delta)\)-regular ultrafilter on \(\delta\), and so it witnesses
        that \(\kappa\) is \(\delta\)-strongly compact.

        \textit{\ref{item:club} implies \ref{item:bagmag}.} 
        Trivial.

        \textit{\ref{item:bagmag} implies \ref{item:cover2}.} 
        Let \(U\) be a countably complete ultrafilter on \(\delta\)
        such that \(U\cap \HOD\) witnesses that \(\kappa\) is \((\omega_1,\delta)\)-strongly
        compact in \(\HOD\). Let \(f : \delta\to P_\kappa(\delta)\cap \HOD\) 
        push \(U\cap \HOD\) forward to a fine ultrafilter in \(\HOD\). 
        Let \(\mathcal U = f_*(U)\). Then \(\mathcal U\) is a fine countably complete
        ultrafilter, and therefore for all \(\sigma\in P_{\omega_1}(\delta)\),
        the set \(\{\tau \in P_\kappa(\delta) : \sigma\subseteq \tau\}\)
        belongs to \(\mathcal U\). Since \(P_\kappa(\delta)\cap \HOD\in \mathcal U\),
        it follows that there is some \(\tau\in P_\kappa(\delta)\cap \HOD\)
        such that \(\sigma\subseteq \tau\). This shows that \(\HOD\) has the
        \((\omega_1,\delta)\)-cover property.

        \textit{\ref{item:cover2} implies \ref{item:sub}.} Trivial.

        \textit{\ref{item:sub} implies \ref{item:stat}.} Trivial.
    \end{proof}
\end{thm}
In the context of \(\AD^+ + V = L(P(\mathbb R))\), 
Woodin's \textit{\(\HOD\)-ultrafilter conjecture} asserts that every countably
complete ultrafilter of \(\HOD\cap V_\Theta\)
extends to a countably complete ultrafilter.
\begin{thm}\label{thm:hod_uf}
    Assume \(\AD^+\), \(V = L(P(\mathbb R))\), and the \(\HOD\)-ultrafilter conjecture.
    Suppose \(\kappa < \Theta\) is a regular cardinal and \(\delta\geq \kappa\) is a
    \(\HOD\)-regular ordinal. Then \(\kappa\) is \(\delta\)-strongly compact in \(\HOD\)
    if and only if \(\kappa\) is \((\omega_1,\delta)\)-strongly compact in \(\HOD\).\qed
\end{thm}

\subsection{Weak extender models}\label{section:wem}
\begin{defn}
    If \(U\) is a \(\kappa\)-complete ultrafilter on \(\lambda\)
    and \(\sigma\in P_\kappa(P(\lambda))\), then
    \(A_U(\sigma) = \bigcap_{A\in U\cap \sigma} A\) and \(\chi_U(\sigma) = \min A_U\).
    Suppose \(\mathcal F\) is a filter on \(P_\kappa(P(\lambda))\)
    and \(U,W\) are \(\kappa\)-complete ultrafilters on \(\lambda\).
    Then \(U <_\mathcal F W\)
    if \(\chi_U(\sigma) < \chi_W(\sigma)\)
    for \(\mathcal F\)-almost all \(\sigma\in P_\kappa(P(\lambda))\).
\end{defn}
\begin{lma}
    Suppose \(\mathcal F\) is a filter on \(P_\kappa(P(\lambda))\).
    Then \(<_\mathcal F\) is a
    strict partial order.
    If \(\mathcal F\) is countably complete, then \(<_\mathcal F\)
    is wellfounded. 
    If \(\mathcal F\) is an ultrafilter, then \(<_\mathcal F\) is linear.
\end{lma}
Thus \(<_\mathcal F\) is a wellorder if 
\(\mathcal F\) is a countably complete ultrafilter
on \(P_\kappa(P(\lambda))\).
The Ultrapower Axiom, which we will not discuss here,
implies that if 
\(\mathcal F\) is the closed unbounded filter on \(P_{\omega_1}(P(\lambda))\),
then \(<_\mathcal F\) is a wellorder.
This is also a consequence of \(\AD_\mathbb R\)
if \(\lambda < \Theta\).

We will use the order \(<_\mathcal F\) in the proof of the following theorem.
\begin{thm}\label{thm:min_appx}
    Suppose \(\kappa\) is strongly compact and \(N\) is an inner model of \ZFC\
    with the \(\kappa\)-cover property. Then there is a minimum extension of \(M\)
    to a model of \ZFC{ }with the \(\kappa\)-approximation property.
\end{thm}
The proof uses the following lemmas.
\begin{lma}\label{lma:appx_coding}
    Suppose \(\delta\) is a regular cardinal,
    \(\kappa\) is \(\delta\)-strongly compact,
    \(N\) is an inner model of \ZFC\ with the \(\kappa\)-cover property,
    and \(M\) is an inner model containing \(N\) such that every \(\kappa\)-complete
    \(N\)-ultrafilter on \(\delta\)
    belongs to \(M\). Then every subset of \(\delta\) that is \(\kappa\)-approximated by 
    \(N\) belongs to \(M\).\qed
\end{lma}

\begin{thm}[Hamkins]\label{thm:hamkins_uniqueness}
    Suppose \(\kappa < \lambda\) are cardinals, \(\cf(\lambda)\geq \kappa\)
    and \(M\) and \(M'\) are inner models with the \(\kappa\)-approximation and cover
    properties below \(\lambda\) such that \(P_\kappa(\kappa^+)\cap M = P_\kappa(\kappa^+)\cap M'\).
    Then \(\Pbd(\lambda)\cap M = \Pbd(\lambda)\cap M'\).\qed
\end{thm}
    \begin{proof}[Proof of \cref{thm:min_appx}]
        Suppose \(\mathcal U\) is a \(\kappa\)-complete fine ultrafilter 
        on \(P_\kappa(\lambda)\) where \(\lambda\) is a 
        strong limit cardinal \(\lambda\) of cofinality at least \(\kappa\)
        and \(P_\kappa(\lambda)\cap N\in \mathcal U\). Suppose
        \(f : \lambda \to \Pbd(\lambda)\cap N\) is a surjection in \(N\).
        Let \(\tilde{\mathcal U}\) denote the pushforward of \(\mathcal U\) by the
        function \(\tilde f: P_\kappa(\lambda)\to P_\kappa(\Pbd(\lambda))\) 
        given by \(\tilde f(\sigma) = f[\sigma]\).
        Let \(\vec U = \vec U_{\mathcal U,f} = \langle U_\alpha : \alpha < \lambda\rangle\) enumerate
        the \(\kappa\)-complete \(N\)-ultrafilters on ordinals less than \(\lambda\) 
        in the wellorder \(<_{\tilde{\mathcal U}}\). Let 
        \(N'\) be the inner model \(L[\vec U,f]\) and let
        \(f'\) be 
        the increasing enumeration of \(\Pbd(\lambda)\cap N'\)
        in the following order:
        for \(a,b\in \Pbd(\lambda)\cap N'\), set
        \(a < b\) if either
        \(\sup(a) < \sup(b)\)
        or 
        \(a\) precedes \(b\) in the canonical wellorder of \(L[\vec U,f]\)
        and \(\sup(a) = \sup(b)\).

        Note that any model \(M\) with the \(\kappa\)-approximation property
        that contains \(N\) 
        must contain the set of \(\kappa\)-complete \(N\)-ultrafilters on ordinals less than \(\lambda\)
        and the order on it induced by \(\tilde {\mathcal U}\). Hence \(N'\subseteq M\).

        Iterating this procedure yields, for each ordinal \(\gamma\), an inner model \(N_\gamma
        = N_{\mathcal U,f,\gamma}\) of ZFC,
        a function \(f_\gamma\), and a \(\lambda\)-sequence \(\vec U_\gamma\) 
        enumerating the \(\kappa\)-complete
        \(N_\gamma\)-ultrafilters on ordinals less than \(\lambda\). 
        Specifically (although still somewhat 
        informally), let \(N_\gamma = L[\langle \vec U_\xi,f_\xi : \xi < \gamma\rangle]\), let
        \(f_\gamma : \lambda\to \Pbd(\lambda)\cap N_\gamma\) be 
        the increasing enumeration of \(\Pbd(\lambda)\cap N_\gamma\)
        in a wellorder similar to the one described in the first paragraph, 
        and finally define 
        \(\vec U_\gamma = \vec U_{\mathcal U,f_\gamma}\) 
        as in the first paragraph.

        The uniformity of this procedure guarantees that any model \(M\) such that
        \(N\subseteq M\) and \(\mathcal U\cap M\in M\) contains \(N_\gamma\)
        for all ordinals \(\gamma\).

        The sequence \(\langle N_\gamma : \gamma < \kappa^+\rangle\) is increasing, 
        so let \(\gamma = \gamma_{\mathcal U,f}\) be the least ordinal
        \(\alpha < \kappa^+\) such that 
        \(N_\alpha\cap \Pbd(\kappa) = N_{\alpha+1}\cap \Pbd(\kappa)\).
        We claim that \(N_{\gamma+1}\) has the \(\kappa\)-approximation property below \(\lambda\).

        Notice that \(P_\kappa(\lambda)\cap N_\gamma = P_\kappa(\lambda)\cap N_{\gamma+1}\):
        if \(\sigma\in P_\kappa(\lambda)\cap N_{\gamma+1}\), then 
        there is some \(\tau\in P_\kappa(\lambda)\cap N\)
        such that \(\sigma\subseteq \tau\); 
        then \(\sigma\in N_\gamma\) and since 
        \(\Pbd(\kappa) \cap N_{\gamma} = \Pbd(\kappa)\cap N_{\gamma+1}\),
        \(P(\tau)\cap N_{\gamma} = P(\tau)\cap N_{\gamma+1}\); the latter set contains \(\sigma\), and hence 
        \(\sigma\in N_{\gamma}\) as desired. 
        \cref{lma:appx_coding} and the definition of \(N_{\gamma+1}\), 
        every bounded subset of \(\lambda\) that is \(\kappa\)-approximated by 
        \(N_\gamma\) belongs to \(N_{\gamma+1}\). 
        Since \(P_\kappa(\lambda)\cap N_\gamma = P_\kappa(\lambda)\cap N_{\gamma+1}\),
        every subset of \(\lambda\) that is \(\kappa\)-approximated by 
        \(N_{\gamma+1}\) is \(\kappa\)-approximated by \(N_{\gamma}\).
        Thus \(N_{\gamma+1}\) has the \(\kappa\)-approximation property below \(\lambda\).

        Now for each pair \((\mathcal U,f)\) such that there is a 
        strong limit cardinal \(\lambda\) of cofinality at least \(\kappa\)
        such that \(\mathcal U\) is a 
        \(\kappa\)-complete fine ultrafilters on \(P_\kappa(\lambda)\) 
        and \(f: \lambda \to \Pbd(\lambda)\) is a surjection in \(N\), 
        let \(M_{\mathcal U,f} = N_{\mathcal U,f,\gamma}\cap H(\lambda)\),
        where \(\gamma = \gamma_{\mathcal U,f}\).
        Reiterating what we have proved above, any inner model 
        \(M\) of ZFC such that \(N\subseteq M\) 
        and \(\mathcal U\cap M\in M\) must contain \(M_{\mathcal U,f}\), and
        therefore any inner model \(M\) with the \(\kappa\)-approximation property
        contains \(\bigcup_{\mathcal U,f} M_{\mathcal U,f}\).

        Fix \(S\subseteq P_\kappa(\kappa^+)\) such that
        for a proper class of appropriate
        \(\mathcal U\) and \(f\), \(M_{\mathcal U,f}\cap P_\kappa(\kappa^+) = S\).
        Let \(C\) be the class of pairs \((\mathcal U,f)\) such that 
        \(M_{\mathcal U,f}\cap P_\kappa(\kappa^+)=S\).
        Then by the Hamkins uniqueness theorem (\cref{thm:hamkins_uniqueness}), 
        for all \(u,v\in C\), either
        \(M_u\subseteq M_v\) or \(M_v\subseteq M_u\). 
        Thus \(M  = \bigcup_{u\in C}M_u\) is an inner model of ZFC, 
        and since
        each \(M_u\) has the \(\kappa\)-approximation property below 
        \(\lambda\), \(M\) has the \(\kappa\)-approximation property.
        Finally, \(N\subseteq M\) since for all 
        but a set of \(u\in C\),
        \(N\cap H(\lambda)\subseteq M_{u}\).
    \end{proof}

The construction of the previous theorem 
is much simpler when \(N\) is an inner model that
not only has the \(\kappa\)-cover property (i.e.,
is positive for the fine filter) but also
is positive for the supercompactness filters
\(\mathcal N_{\kappa,\lambda}\), defined for all \(\kappa\leq \lambda\) as
the intersection of the \(\kappa\)-complete normal fine ultrafilters on \(P_\kappa(\lambda)\).
\begin{defn}
    An inner model \(N\) is \textit{\((\kappa,\lambda)\)-supercompact} if \(N\cap P_\kappa(\lambda)\)
    is \(\mathcal N_{\kappa,\lambda}\)-positive;
    \(N\) is \((\kappa,\infty)\)-supercompact if it is \((\kappa,\lambda)\)-supercompact
    for all cardinals \(\lambda\).
\end{defn}
Note that \(N\) is \((\kappa,\lambda)\)-supercompact if and only if \(N\cap P_\kappa(\lambda)\)
belongs to some \(\kappa\)-complete normal fine ultrafilter on \(P_\kappa(\lambda)\),
and if \(N\) is \((\kappa,\lambda)\)-supercompact for some ordinal \(\lambda\), then
\(N\) is \((\kappa,\alpha)\)-supercompact for all \(\alpha < \lambda\).
The notion bears an obvious resemblance to Woodin's weak extender models.
\begin{defn}
    An inner model \(N\) of \ZFC\ is a \textit{weak extender model of
    \(\kappa\) is \(\lambda\)-supercompact} if there is
    a \(\kappa\)-complete normal fine ultrafilter \(\mathcal U\) on \(P_\kappa(\lambda)\)
    such that \(\mathcal U\cap N\in N\) and \(P_\kappa(\lambda)\cap N\in \mathcal U\);
    \(N\) is a \textit{weak extender model of \(\kappa\) is supercompact} if it is a
    weak extender model of \(\kappa\) is \(\lambda\)-supercompact
    for all cardinals \(\lambda\).
\end{defn}
The substantive part of the following characterization of
weak extender models is due to Woodin and Usuba
independently.
\begin{lma}\label{lma:appx_wem}
    An inner model of \ZFC\ is a
    weak extender model of \(\kappa\) is supercompact
    if and only if it is
    \((\kappa,\infty)\)-supercompact and has
    the \(\kappa\)-approximation property.\qed 
\end{lma}
The following theorem shows that the \((\kappa,\infty)\)-supercompact inner models
are precisely the \({<}\kappa\)-closed inner models of weak extender models.
\begin{thm}\label{thm:positive_wem}
    If \(N\) is a 
    \((\kappa,\infty)\)-supercompact inner model of \ZFC , then
    there is a weak extender model of \(\kappa\) is supercompact
    that contains \(N\) as a \({<}\kappa\)-closed inner model.
    \begin{proof}
        The reverse direction is obvious,
        so we focus on the forwards direction.
        Let \(\lambda\) be a cardinal and
        let \(\mathcal U\) be a \(\kappa\)-complete
        normal fine ultrafilter on \(P_\kappa(\lambda)\)
        such that \(P_\kappa(\lambda)\cap N\in N\), and
        let \(j : V\to M\) be the associated ultrapower embedding.
        Let \(W = j(N)\).
        We claim \(W\) contains every subset
        \(A\) of \(\lambda\) that is \(\kappa\)-approximated by \(N\).

        To see this, note that
        since \(W\) has the \(j(\kappa)\)-cover property
        in \(M\), we can find
        \(\sigma\in P_{j(\kappa)}(j(\lambda))\) such that
        \(j[A]\subseteq \sigma\). 
        But \(j(A)\) is
        \(j(\kappa)\)-approximated by \(W\),
        so \(j(A)\cap \sigma\in W\). Since \(P_\kappa(\lambda)\cap N\in \mathcal U\),
        \(j[\lambda]\in W\), and this implies
        \(A = j^{-1}[j(A)\cap \sigma] \in W\), as desired.

        Next, a familiar argument shows that since
        \(N\cap P(\lambda)\subseteq W\), \(N\cap V_\kappa = W\cap V_\kappa\), 
        and \(N\) has the \(\kappa\)-cover property, in fact,
        \(P_\kappa(\lambda)\cap W\subseteq N\). This implies that
        \(W\) has the \(\kappa\)-approximation property below \(\lambda\).

        It follows that for any \(\kappa\)-complete normal fine ultrafilters
        \(\mathcal U\) and \(\mathcal U'\) 
        with \(P_\kappa(\lambda)\cap N\in \mathcal U\cap \mathcal U'\),
        \(j_\mathcal U(N)\cap P(\lambda) = j_{\mathcal U'}(N)\cap P(\lambda)\)
        by the preceding remarks and the Hamkins uniqueness theorem.
        Letting \(X\) be the union of \(j_\mathcal U(N)\cap P(\lambda)\)
        for all \(\mathcal U\) such that
        \(P_\kappa(\lambda)\cap N\in \mathcal U\), it follows that
        \(M = L(X)\) is a model of ZFC with the \(\kappa\)-approximation and cover
        properties and \(N\) is a \({<}\kappa\)-closed innder model
        of \(M\). Since \(N\) is \((\kappa,\infty)\)-supercompact and \(M\)
        has the \(\kappa\)-approximation property, \cref{lma:appx_wem} implies that \(N\) is a
        weak extender model of \(\kappa\) is supercompact.
    \end{proof}
\end{thm}
A cardinal \(\kappa\) is \textit{distributively supercompact}
if for all cardinals \(\lambda\), there is a \(\kappa\)-distributive partial order \(\mathbb P\)
such that \(\kappa\) is \(\lambda\)-supercompact in \(V^\mathbb P\).
\begin{cor}[\HOD\ Hypothesis]\label{thm:hod_wem}
    Suppose \(\kappa\) is supercompact. Then 
    there is a unique weak extender model of \(\kappa\) is supercompact
    that contains \(\HOD\) as a \({<}\kappa\)-closed inner model.
    In particular, \(\HOD\) satisfies that \(\kappa\) is distributively supercompact.\qed
\end{cor}
The following corollary shows that the first-order theory of HOD has an influence on
the question of whether the least supercompact cardinal is supercompact in HOD.
One could actually replace UA in the argument with any \(\Pi_2\) sentence
that implies that every countably complete
ultrafilter on an ordinal is ordinal definable.
\begin{thm}\label{thm:hod_ua}
    Suppose \(\kappa\) is a supercompact cardinal and the \(\HOD\) hypothesis holds.
    If \(V_\kappa\cap \HOD\) satisfies the Ultrapower Axiom, then \(\kappa\) is supercompact
    in \(\HOD\).
    \begin{proof}
        Let \(N\) be the inner model of \cref{thm:hod_wem}. Then \(N\) is definable
        without parameters and \(V_\kappa\cap \HOD = V_\kappa\cap N\preceq_{\Sigma_2} N\)
        since \(\kappa\) is supercompact in \(N\).
        Since UA is \(\Pi_2\), \(N\) satisfies UA. Now working in \(N\),
        we apply the following consequence of UA:
        if \(\kappa\) is supercompact and \(A\) is a set such that \(V_\kappa\subseteq \HOD_A\),
        then \(V = \HOD_A\). Therefore 
        \(N = (\HOD_A)^N\) for any set of ordinals
        \(A\in N\) such that \(V_\kappa\cap N\subseteq (\HOD_A)^N\). 
        Let \(A\in \HOD\) be a set of ordinals
        such that \(V_\kappa\cap \HOD\subseteq L[A]\).
        Then \(A\in N\) and \(N = (\HOD_A)^N\),
        so there is a wellorder of \(N\) definable over \(N\) from \(A\).
        Since \(N\) is definable without parameters and \(A\) is ordinal definable,
        this wellorder of \(N\) is definable from an ordinal parameter.
        Any ordinal definably wellordered transitive class is contained in \(\HOD\),
        so \(N\subseteq \HOD\). Therefore \(\HOD = N\), 
        and so \(\kappa\) is supercompact in \(\HOD\).
    \end{proof}
\end{thm}
\bibliographystyle{plain}
\bibliography{Bibliography}

\begin{thebibliography}{1}

\bibitem{chengfriedmanhamkins}
Yong Cheng, Sy-David Friedman, and Joel~David Hamkins.
\newblock Large cardinals need not be large in {HOD}.
\newblock {\em Annals of Pure and Applied Logic}, 166(11):1186 -- 1198, 2015.

\bibitem{Jensen}
Keith Devlin and Ronald Jensen.
\newblock Marginalia to a theorem of silver.
\newblock In {\em ISILC Logic Conference}, pages 115--142. Springer, 1975.

\bibitem{UA}
Gabriel Goldberg.
\newblock {\em The {U}ltrapower {A}xiom}.
\newblock PhD thesis, Harvard University, 2019.

\bibitem{UniqueEmbeddings}
Gabriel Goldberg.
\newblock The uniqueness of elementary embeddings.
\newblock To appear.

\bibitem{Solovay}
Robert~M. Solovay.
\newblock Strongly compact cardinals and the {GCH}.
\newblock In {\em Proceedings of the {T}arski {S}ymposium ({P}roc. {S}ympos.
  {P}ure {M}ath., {V}ol. {XXV}, {U}niv. {C}alifornia, {B}erkeley, {C}alif.,
  1971)}, pages 365--372. Amer. Math. Soc., Providence, R.I., 1974.

\bibitem{Woodin}
W.~Hugh Woodin.
\newblock In search of {U}ltimate-{$L$}: the 19th {M}idrasha {M}athematicae
  {L}ectures.
\newblock {\em Bull. Symb. Log.}, 23(1):1--109, 2017.

\end{thebibliography}
\end{document}